\documentclass[a4paper,12pt,reqno]{article}

\usepackage{hyperref}
\hypersetup{
	colorlinks=true,
	linkcolor=blue,
	citecolor=red,
}
\usepackage[dvipsnames]{xcolor}

\usepackage{amsfonts,amsmath,amssymb,amsfonts,amsthm}
\usepackage[left=2.6cm,right=2.6cm,top=2.5cm,bottom=2.5cm,twoside]{geometry}

{%\theoremstyle{cupthm}
\newtheorem{thm}{Theorem}[section]
\newtheorem{prop}[thm]{Proposition}
\newtheorem{cor}[thm]{Corollary}
\newtheorem{lemma}[thm]{Lemma}
}
{\theoremstyle{definition}

\newtheorem{example}[thm]{Example}
}

\newcommand{\veps}{\varepsilon}

\newcommand{\N}{\mathbb{N}}

\newcommand\ud[2]{
	\ensuremath{\overline{#1\!\operatorname{-dens}}(#2)}
}
\newcommand\ld[2]{
	\ensuremath{\underline{#1\!\operatorname{-dens}}(#2)}
}

\newcommand\limd[1]{
	\ensuremath{d_{#1}\,\text{-}\lim}
}

\title{On the limits of real-valued functions in sets involving $ \psi $-density, and applications}
\author{Janne Heittokangas, Zinelaabidine Latreuch, Jun Wang \\ and Amine Zemirni}
\date{}

\begin{document}
\maketitle

\begin{abstract}
	We prove new results on upper and lower limits of real-valued functions by means of $\psi$-densities introduced by P.~D.~Barry in 1962. This allows us to improve several  existing results on the growth of non-decreasing and unbounded real-valued functions in sets of positive density.
	The $\psi$-densities are also used to introduce a new concept of a limit for real-valued functions.  The results in this paper are of interest in real analysis as well as in the theory of meromorphic functions. \\
	
\end{abstract}

\section{Introduction}

For a real-valued function, the passage from the lower/upper limit to the usual limit is valid through a sequence. In many cases, the set, where the aforementioned passage is valid, turns to be larger than just a sequence. In fact, the problem of finding the size of this set is very interesting in the theory of meromorphic functions, where the order and lower order of growth are involved. Recall that order and lower order of a non-decreasing and unbounded function $T$, are given, respectively, by
\begin{equation*}
\overline{\rho}(T) = \limsup_{r\to\infty} \frac{\log T(r)}{\log r} \quad \text{and} \quad \underline{\rho}(T) = \liminf_{r\to\infty} \frac{\log T(r)}{\log r}.
\end{equation*}
Clearly $\underline{\rho}(T)\leq \overline{\rho}(T)$ holds in general. If so preferred, $T(r)$ can be replaced, for example, with the Nevanlinna characteristic $T(r,f)$ of a meromorphic function $f$. The order $\overline{\rho}(f)$ and the lower order $\underline{\rho}(f)$ of $f$ are defined by $\overline{\rho}(f)=\overline{\rho}(T(r,f))$ and $\underline{\rho}(f)=\underline{\rho}(T(r,f))$. For entire functions, $T(r,f)$ can be replaced with the logarithmic maximum modulus $\log M(r,f)$.

It is known that, for any fixed $\ell$ and $L$ satisfying $0\leq \ell \leq L\leq\infty$ there exists a meromorphic function $f$ of order $\overline{\rho}(f)=L$ and of lower order $\underline{\rho}(f)=\ell$ \cite[p.~238]{GO}. If $\ell<L$,
then solving the above problem allows us to know the size of sets of $r$-values, where $T(r,f)$ is near maximal~or~near minimal.
Related to this end, we recall Theorem~\ref{thm-A} from \cite[Corollary~3.7]{LHY}, where the size of such sets $D\subset [1,\infty)$ is measured in terms of upper and lower (linear) densities given by
$$
\overline{\operatorname{dens}}(D)
=\limsup_{r\to\infty}\frac{1}{r} \int_{D\cap [1,r]}dt
\quad\textnormal{and}\quad
\underline{\operatorname{dens}}(D)
=\liminf_{r\to\infty}\frac{1}{r} \int_{D\cap [1,r]}dt.
$$

\begin{thm}\label{thm-A}
	Let $f$ be a meromorphic function such that $0\leq \underline{\rho}(f)<\overline{\rho}(f)\leq\infty$, and let $\underline{\rho}(f)<a\leq b<\overline{\rho}(f)$. Then the sets
	\begin{equation*}
	E=\{r\geq 1: T(r,f)\leq r^a\}\quad\text{and}\quad
	F=\{r\geq 1: T(r,f)>r^b\}
	\end{equation*}
	are of upper density one and of lower density zero.	
\end{thm}

Due to the logarithms appearing in the definitions of orders, it seems natural to study these growth questions in terms of logarithmic densities. Recall that the upper and lower logarithmic densities  of a set $D\subset [1,\infty)$ are given by
$$
\overline{\operatorname{logdens}}(D)
=\limsup_{r\to\infty}\frac{1}{\log r}\int_{D\cap [1,r]}\frac{dt}{t}
\quad\textnormal{and}\quad
\underline{\operatorname{logdens}}(D)
=\limsup_{r\to\infty}\frac{1}{\log r}\int_{D\cap [1,r]}\frac{dt}{t}.
$$
The connection between linear and logarithmic densities is apparent
from the inequalities
\begin{equation}\label{inequalities}
0\leq\underline{\operatorname{ dens}}(D)\leq\underline{\operatorname{\log dens}}(D)\leq\overline{\operatorname{\log dens}}(D)\leq\overline{\operatorname{dens}}(D)\leq 1,
\end{equation}
which can be found in \cite[p.~121]{tsuji}.

The growth questions above are the motivation to study the limit questions for real-valued functions of arbitrary form.
For example, the order and lower order of $ f $ are  the upper limit and the lower limit of a function $ \varphi(r) $ of
the particular form $ \log T(r,f) / \log r $. This calls for a further extension for the concept of density.

We will make use of general $\psi$-densities introduced
by P.~D.~Barry \cite{barry} in 1962. The definitions of these densities and some of their new consequences are discussed
in Section~\ref{Barry-sec}.
This allows us to discuss the upper and lower limits of arbitrary functions $\varphi(r)$ in terms of $\psi$-densities in Section~\ref{limsupinf}. Consequently, results on the growth of unbounded functions $T$ improving Theorem~\ref{thm-A} as well as many other results from the literature are then obtained as corollaries in Section~\ref{growth-sec}.
The $ \psi $-densities are not restricted to growth questions alone but they allow us to introduce a new concept of a limit in $\psi$-density for real-valued functions also. This extends the concept of a limit in density (statistical convergence),
and is used to study the behavior of integrable functions at infinity.  Section~\ref{limit-density} is devoted to presenting new results
in this direction. Limit in $\psi$-density may have potential for further applications in real analysis.

\section{Barry's $\psi$-densities}\label{Barry-sec}

Let $0< r_0<R \le \infty$, and let $ \mathcal{D}(r_0,R) $ denote the class of positive, unbounded, differentiable and strictly increasing functions on $ (r_0,R) $.  Let $ \psi \in \mathcal{D}(r_0,R) $. Then the $ \psi $-measure of a set $ E\subset (r_0,R) $ is the value of the integral
$$
\int_{E} d\psi(t) = \int_E \psi'(t)dt.
$$
Following Barry \cite{barry}, the upper and lower $ \psi $-density of $ E\subset (r_0,R) $ are defined, respectively,~by
$$
\ud{\psi}{E}:= \limsup_{r\to R^-}\frac{1}{\psi(r)}\int_{E\cap[r_0,r)} d\psi(t) \quad \text{and} \quad \ld{\psi}{E}:= \liminf_{r\to R^-}\frac{1}{\psi(r)}\int_{E\cap[r_0,r)} d\psi(t).
$$
It is clear that $  0\le \ld{\psi}{E}\le \ud{\psi}{E}\le 1 $ and
\begin{equation}\label{com}
\ud{\psi}{E^c}+\ld{\psi}{E}=1,
\end{equation}
where $E^c$ denotes the complement of $E$ in $(r_0,R)$. Generalizing the inequalities in \eqref{inequalities}, the $ \psi$-densities and the $ e^\psi $-densities obey the inequalities
\begin{equation}\label{density-r}
0\le \ld{e^\psi}{E} \le \ld{\psi}{E} \le \ud{\psi}{E} \le \ud{e^\psi}{E}\le 1
\end{equation}
for any set $ E \subset (r_0,R) $ and for any $ \psi\in\mathcal{D}(r_0,R) $ by \cite[Lemma~1]{barry}.

The following special cases are certainly of interest. If $ R=+\infty $ and $ \psi(r) = \log r $, then the $ \psi $-densities reduce
to the logarithmic densities, while the $ e^\psi $-densities coincide with the linear densities.
If $ R=1 $ and $ \psi(r) = -\log(1-r)  $, then the $ \psi $-densities reduce to the logarithmic densities.

The following lemma offers a
relation between the $ \psi $-measure and the $ e^\psi $-density, and complements a known result for the logarithmic
measure \cite[p.~9]{zheng}.

\begin{lemma}\label{lem}
	If a set $ E \subset (r_0,R) $ satisfies $\int_E d\psi(t)<\infty$ for $ \psi \in \mathcal{D}(r_0,R) $, then $ \ud{e^\psi}{E} = 0 $.
	In particular, a set $ E \subset (r_0,R) $ of finite logarithmic measure is of zero upper linear density.
\end{lemma}

\begin{proof}
	Let $ \chi_E(t) $ be the characteristic function of the set $ E $, and let $ v(r) = \psi^{-1}(\frac12 \psi(r)) $.
	Since $\psi$ is strictly increasing, $v(r)$ is well-defined and satisfies $v(r)<r$ for all $r>r_0$. Then
	\begin{align*}
	\int_{r_0}^r \chi_E(t) de^{\psi(t)} &=  \int_{r_0}^{v(r)} \chi_E(t) de^{\psi(t)} + \int_{v(r)}^r \chi_E(t) de^{\psi(t)} \\
	&\le \int_{r_0}^{v(r)} de^{\psi(t)} + e^{\psi(r)}\int_{v(r)}^r \chi_E(t) d\psi(t)\\
	& \le e^{\frac{1}{2} \psi(r)} + e^{\psi(r)}\int_{v(r)}^r \chi_E(t) d\psi(t).
	\end{align*}
	As $ r\to R^- $, we have $ v(r) \to R^- $ and hence $ \int_{v(r)}^r \chi_E(t) d\psi(t) \to 0 $. Thus,
	\begin{equation*}
	\ud{e^\psi}{E} = \limsup_{r\to R^-} e^{-\psi(r)}\int_{r_0}^r \chi_E(t) de^{\psi(t)}  =0.
	\end{equation*}
	This completes the proof.
\end{proof}

The following lemma allows us to avoid exceptional sets $ E\subset(r_0,R) $ the size of which with respect to the $ \psi $-measure
is restricted in the sense that $ \ud{\psi}{E}<1 $.

\begin{lemma}\label{exceptional}
	Let $ \psi \in\mathcal{D}(r_0,R) $, and let $ f $ and $ g $ be non-decreasing functions defined on $ (r_0,R) $  satisfying $ f(r) \le g(r) $ for all $ r\in (r_0,R)\setminus E $, where $ \ud{\psi}{E}<1 $. Then, for any $ \alpha> \left(1-\ud{\psi}{E}\right)^{-1} $, there exists an $ r'\in (r_0,R) $ such that $ f(r)\le g(s(r)) $ for all $ r\in (r',R)  $, where $ s(r)=\psi^{-1}(\alpha\psi(r)) $.
\end{lemma}

\begin{proof}
	Suppose there exists an increasing sequence $ (r_n)$  on $ (r_0,R) $ tending to $ R $ such that $ [r_n, s(r_n)] \subset E $ for every $ n\in \N $. Define
	$$
	I=\bigcup_{n=1}^{\infty}\left[r_{n}, s(r_{n})\right].
	$$
	Then $ I\subset E $, but
	\begin{align*}
	\ud{\psi}{I} &\ge \limsup_{n\to\infty}\frac{1}{\psi(s(r_n))} \int_{I\cap[r_0,s(r_n)]} d\psi(t) \\
	& \ge \limsup_{n\to\infty}\frac{1}{\psi(s(r_n))} \int_{[r_n,s(r_n)]} d\psi(t) = 1- \frac{1}{\alpha} > \ud{\psi}{E},
	\end{align*}
	which is a contradiction. Thus, there exists $ r'>r_0 $ such that for any $ r\ge r' $, there exists $ t\in (r,s(r))\setminus E $.  Since $ f $ and $ g $ are non-decreasing, it follows that
	$$
	f(r)\le f(t) \le g(t) \le g(s(r)).
	$$
	This completes the proof.
\end{proof}

Lemma~\ref{exceptional} generalizes \cite[Lemma~3.1]{HK} and \cite[Lemma~3.6]{LHY}, and also extends \cite[Lemma~1.1.7]{zheng}. In
particular, this version works for both finite and infinite intervals.

\section{Upper and lower limits}\label{limsupinf}

Before considering unbounded functions $T$ of finite order or of finite lower order, we proceed to discuss upper and
lower limits in general. The discussion that follows should be of independent interest, and it will be used for proving the growth results on the functions $T$. More precisely,
this section is devoted to prove the following result, which is an improvement of \cite[Theorem~3.2]{HK} in the sense that the present result involves densities rather than measures, and it works for both finite and infinite intervals.

\begin{thm}\label{thm-exceptional}
	Let $ 0<r_0<R\le +\infty $, and let $ \varphi:(r_0,R) \to [0,\infty)$ be a function  with
	$$
	\limsup_{r\to R^-} \varphi(r) = K \quad \text{and} \quad \liminf_{r\to R^-} \varphi(r) = k,
	$$
	where $ 0\le k < K<\infty $. If there exists a $ \psi \in\mathcal{D}(r_0,R)$  such that $ \varphi(r)\psi(r) $ is non-decreasing on $ (r_0,R) $, then for any $ \varepsilon>0 $, the sets
	$$
	F_\veps = \left\{r\in (r_0,R): |\varphi(r)-K|<\varepsilon \right\} \quad \text{and} \quad G_\veps = \left\{r\in (r_0,R): |\varphi(r)-k|<\varepsilon \right\}
	$$
	satisfy
	\begin{eqnarray}
	\ud{\psi}{F_\veps}\ge \frac{\veps}{K}, \quad\, \quad && \quad  \ud{\psi}{G_\veps}\ge \frac{\veps}{k+\veps}, \label{densities}\\
	\ld{\psi}{F_\veps}\le \frac{k}{k+\veps}, \quad && \quad  \ld{\psi}{G_\veps}\le \frac{K-\veps}{K}.\label{densities2}	
	\end{eqnarray}
	In addition,
	\begin{align*}
	\ud{e^\psi}{F_\veps} =\ud{e^\psi}{G_\veps}= 1 , \quad \ld{e^\psi}{F_\veps} =\ld{e^\psi}{G_\veps}= 0.
	\end{align*}						
	If $ K = +\infty $, then for every large $ M>0 $, the set $ H_M = \left\{r\in (r_0, R) : \varphi(r)>M\right\} $ satisfies
	\begin{equation}\label{infinie}
	\ud{\psi}{H_M } =1 \quad \text{and} \quad \ld{\psi}{H_M } \le \frac{k}{M}.
	\end{equation}
\end{thm}

\begin{proof}%[Proof of thm~\ref{thm-exceptional}]
	We prove the first inequality in \eqref{densities}. From the definition of $ \limsup $, there exists an $ r_1\in (r_0,R) $ such that $ \varphi(r)<K+\varepsilon $ for all $ r\in (r_1,R) $. Hence  the sets $ F_\veps $ and $F = \left\{r\in (r_0,R) :\varphi(r)>K-\varepsilon  \right\} $ have the same $ \psi $-density. It suffices to prove now that $ \ud{\psi}{F}\ge \varepsilon/K $. Assume  that $ \ud{\psi}{F}< \varepsilon/K $.
	Here, we might assume that $ \veps < K $.  \\
	Let $ \varepsilon' $ and $ \alpha $ satisfy
	$$
	0< \veps'< \frac{\veps-K \cdot \ud{\psi}{F}}{2-\ud{\psi}{F}} \quad \text{and} \quad \alpha = \frac{K-\veps'}{K-\veps+\veps'}.
	$$
	Then $ \alpha> (1-\ud{\psi}{F})^{-1} $.
	We have  $ \varphi(r)\le K-\varepsilon $ for $r\notin F $.  Then $ \varphi(r)\psi(r)\le (K-\varepsilon)\psi(r) $ for $r\notin F $. Using Lemma~\ref{exceptional} with $ f(r)= \varphi(r)\psi(r)$ and $ g(r)= (K-\varepsilon)\psi(r)$  yields
	\begin{equation}\label{eqeq}
	\varphi(r)\psi(r)\le (K-\varepsilon)\psi(s(r)) , \quad r\in (r',R),
	\end{equation}
	where $ s(r)=\psi^{-1}(\alpha \psi(r)) $. Thus,
	\begin{align*}
	K &= \limsup_{r\to R^-} \varphi(r) =\limsup_{r\to R^-} \frac{\varphi(r) \psi(r)}{\psi(r)} \\
	&\le (K-\veps)\limsup_{r\to R^-} \frac{\psi(s(r))}{\psi(r)}=(K-\veps) \alpha <K-\veps',
	\end{align*}
	which is a contradiction. Hence $ \ud{\psi}{F_\veps} = \ud{\psi}{F}\ge \veps/K $.
	
	Next, we prove the second inequality in \eqref{densities}.  From the definition of $ \liminf $, there exists an $ r_1\in (r_0,R) $ such that $ \varphi(r)>k-\varepsilon $ for all $ r\in (r_1,R) $. Hence  the sets $ G_\veps $ and 	$G = \left\{r\in (r_0,R) :\varphi(r)<k+\varepsilon  \right\} $	have the same $ \psi $-density.  Therefore, it suffices to prove that $ \ud{\psi}{ G}\ge \varepsilon/(k+\veps) $. Assume that $ \ud{\psi}{ G}< \varepsilon/(k+\veps) $. Now let $ \varepsilon' $ and $ \alpha $ satisfy
	$$
	0< \veps'< \frac{\veps-(k+\veps) \ud{\psi}{ G}}{2-\ud{\psi}{ G}} \quad \text{and} \quad  \alpha = \frac{k+\veps-\veps'}{k+\veps'}.
	$$
	Then $ \alpha> ({1-\ud{\psi}{ G}})^{-1} $. We have $ (k+\varepsilon)\psi(r) \le \varphi(r)\psi(r) $ for $r\notin G $. Using Lemma~\ref{exceptional} with $ f(r)= (k+\veps)\psi(r)$ and $ g(r)=\varphi(r)\psi(r)$  yields
	\begin{equation*}\label{eqeq2}
	(k+\veps)\psi(r)\le \varphi(s(r))\psi(s(r)) , \quad r\in (r',R),
	\end{equation*}
	where $ s(r)=\psi^{-1}(\alpha \psi(r)) $. Thus
	\begin{align*}
	k+\veps &\le \liminf_{r\to R^- } \varphi(s(r)) \limsup_{r\to R^-}  \frac{\psi(s(r))}{\psi(r)} \\
	& = \alpha \liminf_{s(r)\to R^- } \varphi(s(r)) = \alpha k < k+\veps-\veps',
	\end{align*}
	which is a contradiction. Hence $ \ud{\psi}{ G_\veps} = \ud{\psi}{ G}\ge \veps/ (k+\veps) $.
	
	To prove the first inequality in \eqref{densities2}, we first claim that there exists an $ r^*\in (r_0,R) $ such that $ F_\veps \cap G_\veps\subset (r_0,r^*) $. To prove this claim, assume the contrary that there exists an increasing sequence $ (r_n) $  on  $ F_\veps \cap G_\veps $ such that $ r_n\to R^- $ as $ n\to\infty $. Then
	$$
	|K-k| \le |f(r_n)-k| + |f(r_n)-K| \to 0 , \quad n\to \infty,
	$$
	and this leads to $ K=k $, which is a contradiction. Thus the claim is true.  Therefore,
	\begin{equation}\label{inclusion}
	F_\veps \subset G_\veps^c \cup (r_0,r^*).
	\end{equation}
	Since $ \ud{\psi}{(r_0,r^*)} = 0 $, it follows that
	$$
	\ld{\psi}{F_\veps} \le \ld{\psi}{G_\veps^c} =  1-\ud{\psi}{G_\veps} \le 1- \frac{\veps}{k+\veps} =\frac{k}{k+\veps}.
	$$
	Similarly, we get the second inequality $ \ld{\psi}{G_\veps} \le (K- \veps)/K $ in \eqref{densities2}.

	Now, assume that $ \ud{e^\psi}{F_\veps}<1 $. Set $ \psi_1(r) = e^{\psi(r)} $.  We then have $ \varphi(r)\le K-\varepsilon $ for $r\notin F_\veps $ and $ \ud{\psi_1}{F_\veps}<1 $ by assumption.  Therefore, \eqref{eqeq} holds with $ s(r)=\psi_1^{-1}(\alpha \psi_1(r)) $. Thus
	\begin{align*}
	K &= \limsup_{r\to R^-} \varphi(r) =\limsup_{r\to R^-} \frac{\varphi(r) \psi(r)}{\psi(r)} \\
	&\le (K-\veps)\limsup_{r\to R^-} \frac{\psi(s(r))}{\psi(r)}\\
	&=(K-\veps)\limsup_{r\to R^-} \frac{\psi(r) + \log \alpha}{\psi(r)} = K-\veps,
	\end{align*}
	which is a contradiction. Hence $ \ud{e^\psi}{F_\veps} =1$.
	Similarly, we can prove $ \ud{e^\psi}{G_\veps} =1 $.
	The equalities $ \ld{e^\psi}{F_\veps}=\ld{e^\psi}{G_\veps} =0$ follow from \eqref{inclusion}.

	Finally, to prove the inequalities in \eqref{infinie}, we assume first that  $\ud{\psi}{H_M} <1$.
	Then $ \varphi(r)\psi(r)\le M \psi(r) $ for $ r\notin H_M $. By Lemma~\ref{exceptional}, we obtain $ \varphi(r)\le \alpha M $ for every $ r$ near $ R $, which is a contradiction. Thus the first inequality in \eqref{infinie} holds. Next, to prove the second inequality in \eqref{infinie}, we see that, similarly to the set $ G $ above, the set
	$$
	H_M^c = \left\{r\in(r_0,R): \varphi(r) \le k + (M-k) \right\}
	$$
	satisfies $ \ud{\psi}{H_M^c} \ge (M-k)/M $. Thus $ \ld{\psi}{H_M} \le~k/M $.
\end{proof}

\begin{prop}\label{propo}
	For every $ \veps \in \left(0, \frac{K-k}{2}\right) $, the inequalities in \eqref{densities} and \eqref{densities2} can be improved to
	\begin{eqnarray}
	\ud{\psi}{F_\veps}\ge 1-\frac{k+\veps}{K},  \quad && \quad  \ud{\psi}{G_\veps}\ge 1-\frac{k}{K-\veps}, \label{densities3}\\
	\ld{\psi}{F_\veps}\le \frac{k}{K-\veps},\quad \ \quad && \quad  \ld{\psi}{G_\veps}\le \frac{k+\veps}{K}.\label{densities4}	
	\end{eqnarray}
\end{prop}

\begin{proof}
	We prove only the first inequality in \eqref{densities3}, and the rest of the inequalities follow similarly. We use the previous inequalities in  \eqref{densities} and \eqref{densities2}.  From the proof of Theorem~\ref{thm-exceptional}, we know that the sets $ F_\veps $ and  $F = \left\{r\in (r_0,R) :\varphi(r)>K-\varepsilon  \right\} $ have the same $ \psi $-density, and the sets $ G_\veps $ and  $G = \left\{r\in (r_0,R) :\varphi(r)<k+\varepsilon  \right\} $ have the same $ \psi $-density.  Then, from the second inequality in \eqref{densities2}, the set
	$$
	F^c = \left\{r\in(r_0,R) : \varphi(r) \le k + (K-k-\varepsilon)\right\}
	$$
	satisfies
	$$
	\ld{\psi}{F^c} \le \frac{K- (K-k-\varepsilon)}{K} = \frac{k+\varepsilon}{K}  .
	$$
	Then $ \ud{\psi}{F_\veps} =  \ud{\psi}{F} \ge 1-(k+\veps)/K $.
\end{proof}

If two functions $ \varphi_1,\varphi_2: [r_0,R) \to [0,\infty)$ satisfy $ \varphi_1(r)<\varphi_2(r) $, then clearly
$$
\limsup_{r\to R^{-}} \varphi_1(r) \le \limsup_{r\to R^{-}} \varphi_2(r).
$$
Conversely, if
$$
\limsup_{r\to R^{-}} \varphi_1(r) < \limsup_{r\to R^{-}} \varphi_2(r),
$$
then what can be said about the size of the set $ \{{r\in (r_0,R)} : \varphi_1(r)<\varphi_2(r)\} $?  The next consequence of Theorem~\ref{thm-exceptional} gives the size of this set by means of the $ \psi $-density.

\begin{cor}\label{coro}
	Let $ \varphi_1, \varphi_2:(r_0,R) \to [0,\infty)$ be functions defined on $ (r_0,R) $ and satisfying
	\begin{equation}\label{comparing}
	\limsup_{r\to R^-} \varphi_1(r) = k_1<  k_2 =\limsup_{r\to R^-} \varphi_2(r),
	\end{equation}
	and let $ \psi \in \mathcal{D}(r_0,R)$ be such that $ \varphi_2(r)\psi(r) $ is non-decreasing on $ (r_0,R) $.
	Then $ \varphi_1(r)<\varphi_2(r) $ holds in a set $ G \subset(r_0,R) $ with $ \ud{\psi}{G}\ge 1- k_1/k_2 $ and $ \ud{e^\psi}{G}= 1 $. The same conclusions hold if $ \limsup $ is replaced with $ \liminf $ on both sides of \eqref{comparing}.
\end{cor}

\begin{proof}
	Suppose first that $ k_2 <\infty $. Let $ 0<\veps< k_2-k_1 $. By the definition of $ \limsup $,
	\begin{equation}\label{coro11}
	\varphi_1(r) < k_1+\veps=k_2-\delta(\veps)
	\end{equation}
	holds for all $ r>r_1>r_0 $, where $ \delta(\veps) = k_2-k_1-\veps>0 $. By Theorem~\ref{thm-exceptional},
	\begin{equation}\label{coro12}
	\varphi_2(r)> k_2-\delta(\veps)
	\end{equation}
	holds in a set $ G^* $ with $ \ud{\psi}{G^*}\ge \delta(\veps)/k_2 $ and $ \ud{e^\psi}{G^*}=1 $. From \eqref{coro11} and \eqref{coro12}, the set
	$$
	G=\left\{r>r_0 : \varphi_1(r)<\varphi_2(r) \right\}
	$$
	satisfies $ \ud{\psi}{G}\ge \delta(\veps)/k_2 $ and $\ud{e^\psi}{G}=1 $. Moreover, since $ G $ is independent on~$ \veps $, it follows by letting $ \veps \to 0^+ $ that $ \ud{\psi}{G}\ge (k_2-k_1)/k_2 $.  If $ k_2 = \infty $, then similarly to the last part of the proof of Theorem~\ref{thm-exceptional}, we get $ \ud{\psi}{G} =1 $.
	
	We can prove similarly that the same conclusions hold if $ \limsup $ is replaced with $ \liminf $ on both sides of \eqref{comparing}. The details are omitted.
\end{proof}

\section{Growth of real-valued functions}\label{growth-sec}

The results in this section are direct consequences of the results in Section~\ref{limsupinf},
and they can easily be applied to obtain results on the growth of meromorphic functions. We restrict ourselves to study non-decreasing functions on the interval $ [1,\infty) $. For non-decreasing functions on $ [0,1) $, analogous results follow the same way.

\begin{cor}\label{coro1}
	Let $T:\left[ 1,\infty\right) \rightarrow (0,\infty)$ be a non-decreasing unbounded function of order $L=\overline{\rho}(T)$ and
	of lower order $\ell=\underline{\rho}(T)$. If $ \ell<a\le b<L $, then the sets
	$$
	H=\{r\geq 1: T(r)\leq r^a\}\quad\text{and}\quad
	I=\{r\geq 1: T(r)>r^b\}
	$$
	satisfy	
	\begin{eqnarray}\label{1}
	\overline{\operatorname{logdens}}(H)\geq \max\left\{\frac{a-\ell}{a},\frac{L-a}{L+\ell-a}\right\},\quad
	\underline{\operatorname{logdens}}(I)\leq \min\left\{\frac{\ell}{b}, \frac{\ell}{L+\ell-b}\right\}.
	\end{eqnarray}
	\begin{eqnarray}\label{2}
	\underline{\operatorname{logdens}}(H)\leq \min\left\{\frac{a}{L},\frac{L+\ell-a}{L}\right\}, \quad \, \quad  \overline{\operatorname{logdens}}(I)\geq \max\left\{\frac{L-b}{L},\frac{b-\ell}{L}\right\}.	
	\end{eqnarray}
\end{cor}
\begin{proof}
	To prove the first inequality in \eqref{1}, we apply Theorem~\ref{thm-exceptional} and Proposition~\ref{propo} with	
	$$
	\varphi(r) = \frac{\log T(r)}{\log r}, \quad \psi(r) = \log r, \quad  \varepsilon = a-\ell.
	$$
	To prove the second inequality in \eqref{1}, we notice that the set $I^c$, which is the complement of $I$ in $ [1,\infty) $, satisfies
	$$
	\overline{\operatorname{logdens}}(I^c)\ge  \max\left\{\frac{b-\ell}{b},\frac{L-b}{L+\ell-b}\right\}.
	$$
	Then, from \eqref{com}, the second inequality in \eqref{1} follows.
	
	The inequalities in \eqref{2} can be proved similarly.
\end{proof}

Corollary~\ref{coro1} is an improvement of Theorem~\ref{thm-A}.
Moreover, the second inequality in \eqref{2} improves \cite[Lemma~2.2]{IT},  \cite[Lemma~3]{LW} and \cite[Lemma~2.7]{ZT}.

Two particular consequences of Theorem~\ref{thm-exceptional} can be stated as follows.
\begin{cor}\label{Rod-Rike}
	Let $T:[1,\infty)\to (0,\infty)$ be a non-decreasing function of order $L\in (0,\infty)$, and let $\varepsilon>0$. Then the set
	$$
	K_1=\left\{r\geq 1: r^{L-\varepsilon}\leq T(r)\leq r^{L+\varepsilon}\right\}
	$$
	satisfies $\displaystyle\overline{\operatorname{logdens}}(K_1)\geq \frac{\varepsilon}{L}.$
\end{cor}
\begin{cor}\label{Rod-Rike-2}
	Let $T:[1,\infty)\to (0,\infty)$ be a non-decreasing function of lower order $\ell\in (0,\infty)$, and let
	$\varepsilon>0$. Then the set
	$$
	K_2=\left\{r\geq 1: r^{\ell-\varepsilon}\leq T(r)\leq r^{\ell+\varepsilon}\right\}
	$$
	satisfies $\displaystyle\overline{\operatorname{logdens}}(K_2)\geq \frac{\varepsilon}{\ell + \varepsilon}.$
\end{cor}
Corollary~\ref{Rod-Rike} improves \cite[Corollary~3.3]{HK}, which claims that the set $K_1$ has infinite logarithmic measure. Replacing $ T(r) $ with $ T(r,f) $ for a meromorphic function $ f $, we see that Corollary~\ref{Rod-Rike-2} improves \cite[Lemma~2.2]{ZT}, which claims that the set $K_2$ has infinite logarithmic measure.

Another consequence of Theorem~\ref{thm-exceptional} is stated in terms of the type of growth. Recall that
$$
\tau(T)=\limsup_{r\to\infty}\frac{T(r)}{r^{\rho}}
$$
is the type of $T$ with respect to its order $\rho =\overline{\rho}(T) \in(0,\infty)$.  We give the following improvement of \cite[Corollary~3.4]{HK}, which claims that the set $N_1$ defined below has infinite linear measure.

\begin{cor}\label{Rod-Rike2}
	Let $T:[1,\infty)\to (0,\infty)$ be a non-decreasing function of order $\rho\in (0,\infty)$ and
	of type $\tau\in (0,\infty)$, and let $\varepsilon_0>0$. Then the set
	$$
	N_1=\left\{r\geq 1: (\tau-\varepsilon_0)r^{\rho}\leq T(r)\leq (\tau+\varepsilon_0)r^{\rho}\right\}
	$$
	satisfies $
	\overline{\operatorname{dens}}(N_1) \geq 1-\left(\frac{\tau-\varepsilon_0}{\tau}\right)^{1 / \rho}.
	$
\end{cor}
The conclusion of Corollary~\ref{Rod-Rike2} follows by applying Theorem~\ref{thm-exceptional} with
\begin{equation*}
\varphi(r) = \frac{T(r)^{1/\rho}}{r}, \quad \psi(r) =r, \quad \varepsilon = \tau^{1/\rho} - (\tau-\varepsilon_0)^{1/\rho}.\qedhere
\end{equation*}

Let $f$ be an entire function of order $\rho\in (0,\infty)$ and of type $\tau\in (0,\infty)$ defined with respect to $\log M(r,f)$. Let $\varepsilon>0$. Then \cite[Lemma~8]{TY} claims that the set
$$
N_2=\left\{r\geq 1: (\tau-\varepsilon)r^{\rho}\leq \log M(r,f)\leq (\tau+\varepsilon)r^{\rho}\right\}
$$
has infinite logarithmic measure. It follows from Lemma~\ref{lem} that a set of finite logarithmic measure has zero upper linear density. Hence we see that Corollary~\ref{Rod-Rike2} is an improvement of \cite[Lemma~8]{TY}.

To compare the growth between two functions, we state the following consequence of Corollary~\ref{coro}.

\begin{cor}\label{comparable-growth-thm}
	Let $T_1, T_2 : \left[ 1,\infty\right) \rightarrow (0,\infty)$ be non-decreasing and unbounded functions such that $\xi(T_1)<\xi(T_2)$, where $\xi$ stands for either the order or the lower order, the same order on both sides of the inequality. Let $\phi:[1,\infty)\to(0,\infty)$ be any non-decreasing function such that $\log\phi(r)=o(\log r)$ as $ r\to \infty $. Then the set
	\begin{equation*}\label{F}
	P=\{r\geq 1 : \phi(r) T_1(r)< T_2(r)\}
	\end{equation*}
	satisfies
	$$
	\overline{\operatorname{logdens}}(P)\geq 1-\frac{\xi(T_1)}{\xi(T_2)}\quad \text{and}\quad \overline{\operatorname{dens}}(P)=1.
	$$
\end{cor}

The conclusion of Corollary~\ref{comparable-growth-thm} follows by using Corollary~\ref{coro} with
\begin{equation*}
\varphi_1(r) = \frac{\log T_1(r) + \log \phi(r)}{\log r}, \quad  \varphi_2(r)=\frac{\log T_2(r)}{\log r}  \quad\text{and} \quad  \psi(r)=\log r. \qedhere
\end{equation*}

Corollary~\ref{comparable-growth-thm} improves \cite[Lemma~7]{HILT}.
A special case of Corollary~\ref{comparable-growth-thm} is implicitly proved in \cite[p.~347]{KK} in the case $\overline{\rho}(T_2)<\infty$.

A possible choice for $\phi$ in Corollary~\ref{comparable-growth-thm} is $\phi(r)=(\log r)^\beta$, where $\beta>0$. 	
If we replace $ T_1(r) $ and $ T_2(r) $ by $ T(r,f) $ and $ T(r,g) $, respectively, where $ f $ and $ g $ are meromorphic functions, and if $\phi$ is unbounded, then Corollary~\ref{comparable-growth-thm} states in particular, that
\begin{equation}\label{small-function}
T(r,f)=o(T(r,g)),\quad r\in P.
\end{equation}
Thus $f$ is a small function of $g$ relative to the set $P$. Recall that in the complex function theory, a meromorphic function $f$ is said to be a small function of another meromorphic function $g$, if $T(r,f)=o(T(r,g))$ for all $r$ outside of a set of finite linear measure (or sometimes outside of a set of finite logarithmic measure). Small functions appear frequently in the theories of complex differential and functional equations, which in turn typically rely on growth estimates for logarithmic derivatives and for logarithmic differences. The former estimates are usually valid outside of exceptional sets of finite linear/logarithmic measure, while the exceptional sets in the latter estimates may go up to upper logarithmic density $<\veps$. Hence, in most cases, the definition of small functions could be relaxed to \eqref{small-function}, where the set $P$ is required to have positive logarithmic upper density.

Next, we give a result about comparing the growth of two functions in the case when they have the same order but different types.

\begin{cor}\label{comparable-growth-thm-2}
	Let $T_1, T_2 : \left( 1,\infty\right) \rightarrow (0,\infty)$ be continuous, non-decreasing functions both having order
	$\rho\in (0,\infty)$, and	 $\tau(T_1)<\tau(T_2)$.
	Let $C\in (1,\tau(T_2)/\tau(T_1))$. Then the set
	$$
	Q=\{r\geq 1 : CT_1(r)< T_2(r)\}
	$$
	satisfies
	\begin{equation*}\label{dens2}
	\overline{\operatorname{dens}}(Q)\geq 1-C^{1 / \rho}\left(\frac{\tau(T_1)}{\tau(T_2)}\right)^{1 / \rho}.
	\end{equation*}
\end{cor}
This follows by using Corollary~\ref{coro} with
\begin{equation*}
\varphi_1(r) = \frac{C^{1/\rho} T_1(r)^{1/\rho}}{r}, \quad \varphi_2(r) = \frac{T_2(r)^{1/\rho}}{r} \quad \text{and} \quad \psi(r) =r. \qedhere
\end{equation*}

In \cite[Lemma~4]{H}, it is shown that for a meromorphic function $ f $ of order $ \rho $, and for constants $ C_1>1 $ and $ C_2>1 $, the set
\begin{equation}\label{H}
U= \{r: T(C_1r,f)\ge C_2 T(r,f)\}
\end{equation}
satisfies
\begin{equation}\label{Hi}
\overline{\operatorname{logdens}}(U) \le  \frac{\rho\log C_1}{\log C_2}.
\end{equation}
If either $ \rho=0 $  or $ C_2^{1/\rho} \ge C_1 $, then the inequality \eqref{Hi} is meaningful, and it gives information about
size of the set $U$.
In the opposite case when $ \rho>0 $ and $ C_2^{1/\rho} < C_1 $, the quantity $ \frac{\rho\log C_1}{\log C_2} $ is larger than~1, and hence we may conclude nothing from \eqref{Hi}. In this case, the set $ U $ is expected to be large,
and its size can be estimated directly by means of Corollary~\ref{coro} with an additional assumption on the type of  $ f $. In fact, we have the following result.

\begin{cor}\label{coroH}
	Let $T:[1,\infty)\to (0,\infty)$ be a non-decreasing function of order $\rho\in (0,\infty)$ and of type $\tau\in (0,\infty)$. Let $ C_1>1 $ and $ C_2> 1$ be such that $C_2^{1/\rho} < C_1$. Then the set
	$$
	V=\{r: T(C_1r)\ge C_2 T(r)\}
	$$
	satisfies
	\begin{equation*}\label{Ni}
	\overline{\operatorname{dens}}(V)\geq 1- \frac{C_2^{1/\rho}}{C_1}.
	\end{equation*}
\end{cor}

This follows by using Corollary~\ref{coro} with
\begin{equation*}
\varphi_1(r) = \frac{C_2T(r)^{1/\rho}}{r}, \quad \varphi_2(r)=\frac{T(C_1r)^{1/\rho}}{r}, \quad \psi(r)=r.\qedhere
\end{equation*}

Replacing $ T(r) $ with $ T(r,f) $ in Corollary~\ref{coroH}, where $ f $ is a meromorphic function of order $\rho\in (0,\infty)$ and of type $\tau\in (0,\infty)$, we find that the set $ U $ in \eqref{H}  is large in the sense that $ \overline{\operatorname{dens}}(U) \ge 1- C_2^{1/\rho} / C_1 $.

\section{Behavior of integrable functions at infinity}\label{limit-density}
%\section{Limits in $\psi$-density}\label{limit-density}

Recently, Niculescu and Popovici \cite{NP, NP2} have discussed necessary conditions for the integrability of real-valued
functions based on a concept of limit in linear density. We will generalize this concept for the $\psi$-density, where
$ \psi \in \mathcal{D}(r_0,R) $. We say that a function  $ f:(r_0,R) \to \mathbb{R}$ has a limit $l\in \mathbb{R}$ in $\psi$-density as $r\to R^-$ if the set $\{r\in (r_0,R):|f(r)-l|\ge\varepsilon\}$ has zero $\psi$-density, whenever $\varepsilon>0$. We denote this limit~by
$$
\limd{\psi} _{r \to R^-} f(r)=l.
$$
The value $+\infty$ (resp. $-\infty$) is called the limit of $f$ in $\psi$-density as $r\to R^-$, and we denote~it
$$
\limd{\psi}_{r \to R^-} f(r)= +\infty\quad (\text{resp.}\,-\infty),
$$
if for each  $M\in \mathbb{R}$, the set $\{r\in(r_0,R) : f(r) \leq  M\}$ (resp. $\{r\in(r_0,R):f(r) \geq  M\}$) has zero $\psi$-density.
%The notion of limit in $\psi$-density is not totally knew. Koopman and Neumann \cite{KV} have introduced the concept of limit in density where $\psi(r)=r$ in connection with convergence of certain weighted arithmetic means.
Clearly, if $\lim_{r \to R^-}f(r)=l$, then $\limd{\psi}_{r \to R^-} f(r)=l$ for any $ \psi \in \mathcal{D}(r_0,R) $. The converse is not true in general. For example, the function
\begin{equation}\label{fex}
f(r)=\left\{\begin{array}{ll}
1, & \text { for } r \in\left[n, n+1 / 2^{n}\right],\ n \in \mathbb{N}, \\
0, & \text { otherwise},
\end{array}\right.
\end{equation}
does not have a limit as $ r\to\infty $, but $ \limd{r}\limits_{r\to\infty}f(r)=0 $. 		

%\begin{example}\label{ex}
%Define
%$$
%f(r)=\left\{\begin{array}{ll}
%1, & \text { for } r \in\left[n, n+1 / 2^{n}\right],\ n \in \mathbb{N}, \\
%0, & \text { otherwise}.
%\end{array}\right.
%$$
%Clearly, $\limd{r}_{r \to \infty}f(r)=0$, because the set $$A_{\varepsilon}=\{r:f(r)>\varepsilon\}=\bigcup_{n\geq 1}\left[n, n+1 / 2^{n}\right]$$ satisfies
%	\begin{eqnarray*}
%	\operatorname{dens}(A_{\varepsilon})&=&\limsup_{r \to \infty}\frac{\int_{A_{\varepsilon}\cap\left[ 1,r\right)} dt}{r}
%	\leq\limsup_{n \to \infty}\frac{\int_{A_{\varepsilon}\cap\left[ 1,n\right)} dt}{n-1}\\
%	&\leq& \limsup_{n \to \infty}\frac{\sum_{n=1}^{\infty}\int_{n}^{n+1/2^n}dt}{n-1}=\limsup_{n \to \infty}\frac{\sum_{n=1}^{\infty}1/2^n}{n-1}=0,
%	\end{eqnarray*}
%while $f(r)\not\rightarrow 0$, as $r\to \infty$.
%\end{example}

We prove the following result which gives a general necessary condition for integrable functions.

\begin{thm}\label{Th2}
	Let $ f:[r_0,\infty) \to \mathbb{R}$ be a locally integrable function on $ [r_0,\infty) $. If $\left|\int_{r_0}^{\infty}f(t)\,dt\right|<\infty$ , then for any $ \psi\in\mathcal{D}(r_0,\infty) $ we have
	\begin{equation}\label{lim1}
	\lim_{r \to \infty}\frac{1}{\psi(r)} \int_{r_0}^{r} \psi(t) f(t)\,dt =0.
	\end{equation}
	Moreover, if $\int_{r_0}^{\infty}|f(t)|\,dt<\infty$, then
	\begin{equation}\label{limd}
	\limd{\psi}\limits_{r \to \infty} \frac{\psi(r)}{\psi'(r)}f(r)=0.
	\end{equation}
\end{thm}

\begin{proof}
	Let $\varepsilon>0$ and $ \psi\in \mathcal{D}(r_0,\infty) $. Then there exists an $r_1>r_0$ such that for every $ r>r_1 $,
	$$
	\left| \int_{r_1}^{r}f(t)\,dt\right| <\frac{\varepsilon}{3} \quad  \text{and} \quad \left| \frac{1}{\psi(r)} \int_{r_0}^{r_1}\psi(t)f(t)\,dt\right| <\frac{\varepsilon}{3}.
	$$
	Therefore, for every $ r>r_1 $, we have
	\begin{eqnarray*}
		\left|\frac{1}{\psi(r)} \int_{r_0}^{r} \psi(t)f(t)\,dt \right|&=& \left|\frac{1}{\psi(r)}\left( \int_{r_0}^{r_1} \psi(t)f(t)\,dt + \int_{r_1}^{r}\psi(t)\left( \int_{r_1}^{t}f(s)ds\right)'\, dt\right) \right|\\
		&\le& \left|\frac{1}{\psi(r)} \int_{r_0}^{r_1} \psi(t)f(t)\,dt\right|+ \left| \int_{r_1}^{r} f(s)\,ds \right|+
		\left|\frac{1}{\psi(r)}\int_{r_1}^{r} \psi'(t)\left( \int_{r_1}^{t}f(s)\,ds\right)\,dt\right| \\
		&<& \frac{\veps}{3} + \frac{\veps}{3} + \frac{\psi(r)-\psi(r_1)}{\psi(r)}\frac{\veps}{3}< \veps,
	\end{eqnarray*}
	which results in \eqref{lim1}.
	
	Now, assume that $\int_{r_0}^{\infty}|f(t)|\,dt<\infty$. Let $ \varepsilon>0 $ and $S_{\varepsilon}=\{r>r_0:\frac{\psi(r)}{\psi'(r)}|f(r)|>~\varepsilon\}$. Then, by using \eqref{lim1}, we get
	$$
	0\le \frac{1}{\psi(r)}\int_{S_{\varepsilon}\cap[r_0,r)}\psi'(t)\, dt \leq \frac{1}{\veps\psi(r)} \int_{r_0}^{r} \psi(t)|f(t)|\, dt\to 0, \quad r\to\infty,
	$$
	which means $ \ud{\psi}{S_\veps}=0 $ for every $ \veps>0 $, and hence we get \eqref{limd}.
\end{proof}

The first part of Theorem~\ref{Th2} generalizes \cite[Theorem~0.1]{Mi}, while the second part generalizes \cite[Theorems 3--4]{NP} and \cite[Theorem 2]{NP2}. The first part can be used to show the divergence of $ \int_{r_0}^{\infty}f(t)\,dt $ as follows:
If there exists a $ \psi\in\mathcal{D}(r_0,\infty) $ such that \eqref{lim1} does not hold, then  $\int_{r_0}^{\infty}f(t)\,dt$ diverges.
The following example shows that the first part of Theorem~\ref{Th2} is stronger than \cite[Theorem~0.1]{Mi}.

\begin{example}
	Consider the improper integral
	$$
	I=\int_{2}^{\infty}\frac{\sin^2 t}{t\log t}\,dt.
	$$
	We have, by L'Hospital's rule,
	$$
	\lim_{r\to\infty} \frac{1}{r} \int_2^r  \frac{\sin^2 t}{\log t}\, dt= 0.
	$$
	From \cite[Theorem~0.1]{Mi}, we conclude nothing about the integral $ I $. However, if we take $ f(r) =  \frac{\sin^2 r}{r\log r}$ and $ \psi(r) = \log r $ in \eqref{lim1}, we find
	$$
	\lim_{r \to \infty}\frac{1}{\log r}\int_{2}^{r}\frac{\sin^2t}{t}\,dt
	=\lim_{r \to \infty}\frac{1}{2\log r}(\log r+\operatorname{Ci}(2r))=\frac{1}{2},
	$$
	where $\operatorname{Ci}(r)=-\int_{r}^{\infty}\frac{\cos t}{t}\,dt$ is the cosine integral. Thus, according to Theorem~\ref{Th2}, the improper integral $ I $ diverges.
\end{example}

It is natural to ask whether the limit in $ \psi $-density \eqref{limd} can be improved to the usual limit.
Surprisingly, Theorem~\ref{thm-exceptional} plays a key role in finding a sufficient condition to ensure that  \eqref{limd} is improved to the usual limit.  In fact, we have the following result.
\begin{thm}\label{Th3}
	Let $ f: [r_0,\infty) \to \mathbb{R}$ be a function satisfying $\int_{r_0}^{\infty}|f(t)|\,dt<\infty$. Suppose there exists $ \psi \in \mathcal{D}(r_0,\infty) $ such that one of the following holds:
	\begin{itemize}
		\item[(i)] 	${\psi(r)^2}|f(r)| / {\psi'(r)}$ is non-decreasing on $ (r_1,\infty) $ for some $ r_1\ge r_0 $,
		\item[(ii)] 	${|f(r)|}/{\psi'(r)}$ is non-increasing on $ (r_1,\infty) $ for some $ r_1\ge r_0 $.
	\end{itemize}
	Then
	$$
	\lim_{r \to \infty}\frac{\psi(r)}{\psi'(r)}f(r)=0.
	$$
\end{thm}

The following lemma, which relies on Theorem~\ref{thm-exceptional}, is needed to prove Theorem~\ref{Th3} in the the case when (i) holds.

\begin{lemma}\label{Th 5}
	Let $ 0<r_0<R\le +\infty $, and let $ f:[r_0,R) \to [0,\infty)$. Suppose that there exists a $ \psi \in \mathcal{D}(r_0,R) $ such that $\limd{\psi}\limits_{r \to R^-} f(r)=l$. Then either $\lim\limits _{r \to R^-} f(r)=l$ or the function $f(r) \psi(r)$ is not non-decreasing on $(r_0,R)$.
\end{lemma}

\begin{proof}
	We consider the case $l\in \left[ 0,\infty\right) $ only since the case $l=\infty$ follows similarly. Suppose that there exists a $ \psi \in \mathcal{D}(r_0,R) $ such that $\limd{\psi}_{r \to R^-} f(r)=l$. Moreover, suppose on the contrary to the assertion that $f(r) \psi(r)$ is non-decreasing on $(r_0,R)$ and that $f(r)$ doesn't have a limit as $r\to R^-$, i.e.,
	$$
	k=\liminf_{r\to R^-} f(r)\neq \limsup_{r\to R^-} f(r).
	$$ 	
	
	Let $ \veps>0 $, and let 
$$A_\veps = \left\{r\in (r_0,R): |f(r)-k|<\varepsilon \right\},\quad L_\veps = \left\{r\in (r_0,R): |f(r)-l|\ge\varepsilon \right\}.$$
	First, we prove that $ l\neq k $. Assume on the contrary that $ l = k $. Then, from Theorem~\ref{thm-exceptional} we obtain that
	$$
	\ud{\psi}{L_\veps} = \ud{\psi}{A_\veps^c}= 1-\ld{\psi}{A_\veps}>0,
	$$
	which is a contradiction with the definition of limits in $ \psi $-density. Thus $ l \neq k $.
	
	Next, we prove that $A_{\varepsilon}$ and $L_{\varepsilon}^c$ are disjoint for all $r\geq r^*$, where $r^*\in (r_0,R)$.
	Assume on the contrary that there exists an increasing sequence $ (r_n) $  on  $ A_\veps \cap L_\veps^c $ such that $ r_n\to R $ as $ n\to\infty $. Then
	$$
	|k-l| \le |f(r_n)-l| + |f(r_n)-k| < 2\veps , \quad n\to \infty,
	$$
	and this leads to $ k=l $, which is a contradiction.  Therefore $A_{\varepsilon}\subset L_{\varepsilon}\cup (r_0,r^*)$. It follows from this, Theorem~\ref{thm-exceptional} and the definition of limits in $ \psi $-density,  that
	$$
	0< \ud{\psi}{A_{\varepsilon}}\leq\ud{\psi}{L_{\varepsilon}}+ \ud{\psi}{(r_0,r^*)} =0,
	$$
	which is a contradiction.  Thus, either $\lim\limits _{r \to R^-} f(r)=l$ or $f(r) \psi(r)$ is not non-decreasing on $(r_0,R)$
\end{proof}

\begin{proof}[Proof of Theorem~\ref{Th3}]
	(i) From Theorem~\ref{Th2}, we have
	$$
	\limd{\psi}_{r \to \infty} \frac{\psi(r)}{\psi'(r)} |f(r)|=0.
	$$
	Since $\left[\frac{\psi(r)}{\psi'(r)}|f(r)| \right]\psi(r)$ is non-decreasing, we infer from Lemma~\ref{Th 5} that $\lim\limits_{r \to \infty} \frac{\psi(r)}{\psi'(r)}f(r)=~0$.
	
	(ii) Let $ s(r) = \psi^{-1}(\frac{1}{2}\psi(r)) $. We have for every $ r>\psi^{-1}(2\psi(r_1)) $,
	$$
	\left|\frac{\psi(r)}{\psi'(r)} f(r) \right| = 2\frac{|f(r)|}{\psi'(r)}\int_{s(r)}^{r} \psi'(t)dt \le 2\int_{s(r)}^{r} \frac{|f(t)|}{\psi'(t)}\psi'(t)dt=2\int_{s(r)}^{r} |f(t)|dt
	$$
	from which it follows that $\lim\limits_{r \to \infty} \frac{\psi(r)}{\psi'(r)}f(r)=~0.$
\end{proof}

The following example illustrates Theorem~\ref{Th3}.

\begin{example}\label{ex}
	The function
	$$f(r)=\frac{1}{r(\log r)^{\beta}}, \quad \beta>1,$$
	is integrable on $(e,\infty)$. By taking $\psi(r)=r$ in Theorem~\ref{Th3}, we see that both conditions (i) and (ii) hold and  hence $\lim\limits _{r \to \infty} rf(r)=0$.
	
	If we take $ \psi(r)=(\log r)^\beta $ in Theorem~\ref{Th3}, then we see that both conditions (i) and (ii) hold and  hence $\lim\limits _{r \to \infty} r\log r f(r)=0$.
	
	If we take $ \psi(r) = \log(\log(r))$ in Theorem~\ref{Th3}, then the condition (i) does not hold. Meanwhile, the condition (ii) holds and hence $ \lim \limits_{r \rightarrow \infty} r\log(r){\log (\log (r))} f(r)=0 $.
\end{example}

%One may see that all the obtained results in this sections remain valid on the bounded interval $(r_0,R)$ where $ \psi \in \mathcal{D}(r_0,R) $ and $R$ is the only singularity of the improper integral $\int_{r_0}^{R}f(t)d\psi(t)$. The following example shows that Theorem~\ref{Th3} remains valid under the above conditions.
%\begin{example}
%	The function
%	$$
%	f(r)=\frac{1}{\sqrt{1-r}},
%	$$
%	is integrable on $(0,1)$. By taking $\psi(r)=1/(1-r)$ in Theorem~\ref{Th3}, we see that both conditions of (i) and (ii) are satisfied and hence $\lim_{r \to 1^-}(1-r)f(r)=0$.

%If we take $\psi(r)=-\log(1-r)$ in Theorem~\ref{Th3}, then the condition (i) does not hold. Meanwhile, the condition (ii) holds and hence $\lim_{r \to 1^-}(1-r)\log\left( \frac{1}{1-r}\right) f(r)=0$.
%\end{example}

\section*{Acknowledgment}
Heittokangas wants to thank the School of Mathematical Sciences at the Fudan University for its hospitality during his visit in August 2019. Latreuch was supported by the Directorate General for Scientific Research and Technological Development (DGRSDT) of Algeria.
Wang was supported by the Natural Science Foundation of China (No.~11771090).

\vspace{1cm}
\noindent
{Janne~Heittokangas}, {University of Eastern Finland \\
	Department of Physics and Mathematics,\\
	P.O. Box 111, 80101 Joensuu, Finland}\\
e-mail: \href{mailto:janne.heittokangas@uef.fi}{janne.heittokangas@uef.fi}
\bigskip

\noindent
{Zinelaabidine~Latreuch}, {University of Mostaganem\newline
Department of Mathematics\newline
Laboratory of Pure and Applied Mathematics\\
 B.~P.~227 Mostaganem, Algeria\\
e-mail: \href{mailto:z.latreuch@gmail.com}{z.latreuch@gmail.com}}
\bigskip

\noindent{Jun~Wang}, {Fudan University\newline
 School of Mathematical Sciences\newline
 Shanghai 200433, P.~R.~China} \\
e-mail:  \href{mailto:majwang@fudan.edu.cn}{majwang@fudan.edu.cn}
\bigskip

\noindent
{Amine~Zemirni},{University of Eastern Finland\\
	Department of Physics and Mathematics,\\
	P.O. Box 111, 80101 Joensuu, Finland}\\
e-mail: \href{mailto:amine.zemirni@uef.fi}{amine.zemirni@uef.fi}


\begin{thebibliography}{99}
	
	\bibitem{barry}
	Barry P.~D.,
	\emph{The minimum modulus of small integral and subharmonic functions}. Proc.~London
	Math.~Soc.~(3) {\bf 12} (1962), 445--495.
	
	%	\bibitem{FS}
	%	Fenton P. C. and M. M. Strumia,
	%	\emph{Wiman-Valiron theory in the disc},
	%	 J. Lond. Math. Soc. (2) \textbf{79} (2009), no.~2, 478--496.
	
	\bibitem{GO}
	Goldberg A.~A.~and I.~V.~Ostrovskii,
	\emph{Value Distribution of Meromorphic Functions}.
	Translated from the 1970 Russian original by Mikhail Ostrovskii.
	With an appendix by Alexandre Eremenko and James K.~Langley.
	Translations of Mathematical Monographs, 236. American Mathematical Society, Providence, RI, 2008.

	
	\bibitem{HK}
	Halburd R.~G.~and R.~J.~Korhonen,
	\emph{Nondecreasing functions, exceptional sets and generalized Borel lemmas}.
	J.~Aust.~Math.~Soc.~\textbf{88} (2010), no.~3, 353--361.
	
	\bibitem{H} Hayman~W.~K.,
	\emph{On the characteristic of functions meromorphic in the plane and of their integrals}.
	Proc.~London~Math.~Soc.~\textbf{14A} (1965), 93--128.
	
	\bibitem{HILT}
	Heittokangas J., K.~Ishizaki, I.~Laine and K.~Tohge,
	\emph{Exponential polynomials in the oscillation theory}.
	J.~Differential Equations \textbf{272} (2021), 911--937.	
	
	
	
	\bibitem{IT}
	Ishizaki K.~and K.~Tohge,
	\emph{On the complex oscillation of some linear differential equations}.
	J.~Math.~Anal.~Appl.~\textbf{206} (1997), 503--517.
	
	\bibitem{KK}
	Kwon K.~H.~and J.~H.~Kim,
	\emph{Maximum modulus, characteristic, deficiency and growth of solutions
		of second order linear differential equations}.
	Kodai Math.~J.~\textbf{24} (2001), no.~3, 344--351.

	
	\bibitem{LW}
	Laine I.~and P.~Wu,
	\emph{On the oscillation of certain second order linear differential equations}.
	Rev.~Roumaine Math.~Pures Appl.~\textbf{44} (1999), no.~4, 609--615.
	
	\bibitem{LHY}
	Long J., J.~Heittokangas and Z.~Ye,
	\emph{On the relationship between the lower order of coefficients and the growth of solutions of differential equations}.
	J.~Math.~Anal.~Appl.~\textbf{444} (2016), no.~1, 153--166.
	
	\bibitem{Mi}
	Mihai M.~V.,
	\emph{A remark on the behavior of integrable functions at infinity}.
	An.~Univ.~Craiova Ser.~Mat.~Inform.~\textbf{38} (2011), no.~4, 100--101.
	
	\bibitem{NP}
	Niculescu C.~P.~and F.~Popovici, \emph{A note on the behavior of integrable functions at infinity}.
	J.~Math.~Anal.~Appl.~\textbf{381} (2011) 742--747.
	
	\bibitem{NP2}
	Niculescu C.~P.~and F.~Popovici, \emph{The behavior at infinity of an integrable function}.
	Expo.~Math.~\textbf{30} (2012) 277--282.
	
	\bibitem{tsuji}
	Tsuji M.,
	\emph{Potential Theory in Modern Function Theory}.
	Reprinting of the 1959 original. Chelsea Publishing Co., New York, 1975.
	
	\bibitem{TY}
	Tu J.~and C.-F.~Yi,
	\emph{On the growth of solutions of a class of higher order linear differential equations with coefficients having the same order}.
	J.~Math.~Anal.~Appl.~\textbf{340} (2008), 487--497.
	
	\bibitem{ZT}
	Zheng X.-M.~and J.~Tu,
	\emph{Growth of meromorphic solutions of linear difference equations}.
	J.~Math.~Anal.~Appl.~\textbf{384} (2011), no.~2, 349--356.
	
	
	\bibitem{zheng}
	Zheng J.-H.,\emph{Value Distribution of Meromorphic Functions}. Tsinghua University Press, Beijing, Springer, Heidelberg, 2010.
\end{thebibliography}
\end{document}